\documentclass[12pt]{article}
\usepackage{amsmath}
\usepackage{amsthm}
\usepackage{amsfonts}
\usepackage{amssymb,amscd}
\usepackage{amscd}
\usepackage{tikz}
\usepackage{bbm}

\usepackage{amsmath,latexsym,amssymb,amsthm}
\theoremstyle{plain}
\newtheorem{thm}[subsection]{Theorem}
\newtheorem{lem}[subsection]{Lemma}

\newcommand{\card}[1]{\ensuremath{\left |#1\right |}}

\theoremstyle{definition}

\newtheorem{defn}[subsection]{Definition}

\def\v{\vee}
\def\w{\wedge}

\def\L{\cal{L}}

\def\D{\cal{D}}
\def\e{\varepsilon}
\def\a{\alpha}

\def\b{\beta}

\def\l{\lambda}

\def\d{\delta}
\def\g{\gamma}

\def\v{\vee}
\def\w{\wedge}

\def\L{\cal{L}}

\def\e{\varepsilon}
\def\a{\alpha}

\def\b{\beta}

\def\l{\lambda}

\def\d{\delta}
\def\g{\gamma}

\begin{document}

\title{On singularities of lattice varieties }
\author{
        Himadri Mukherjee\\
                Department of Mathematics and Statistics\\
        Indian Institute of Science Education and Research, Kolkata\\
        himadri@iiserkol.ac.in
       }

\date{\today}

\maketitle

\begin{abstract}
 Toric varieties associated with distributive lattices arise as a fibre of a flat degeneration of a Schubert variety in a minuscule. The singular locus of these varieties has been studied by various authors. In this article we prove that the number of diamonds incident on a lattice point $\a$ in a product of chain lattices is more than or equal to the codimension of the lattice. Using this we also show that the lattice varieties associated with product of chain lattices is smooth.
\end{abstract}

\section{Introduction}

\label{introduction}
The toric varieties associate to distributive lattices are studied by various authors for the last few decades. In \cite{Hibi} Hibi shows that the $k$-algebra $k[\L]$ associated to a lattice $\L$ is an integral domain if and only if the lattice $\L$ is a distributive lattice. Furthermore using a standard monomial basis it was showed that the $k$-algebra is normal. Therefore the toric variety associated to the binomial ideal $I(\mathcal{L})= <x_\a x_\b - x_{\a \v \b} x_{\a \w \b} | \a \, , \b \in \L >$ related to a distributive lattice $\L$ is a normal toric variety. In \cite{g-l} Lakshmibai and Gonciulea shows that the cone over the Schubert variety $X(\omega)$ i.e $\widehat{X(\omega)}$ associated to a $\omega \in \mathrm{I}_{(d,n)}$ degenerates to the lattice variety $X(\L_\omega)$ for the Bruhat poset $\L_\omega$. They also find the orbit decomposition of these lattice varieties and propose conjectures related to their singularities \cite{GLnext}. Wagner in \cite{wagner} finds the singular locus of these varieties depending on conditions on {\em contractions} of the poset of join irreducibles $J$. In \cite{HL} the singular locus of these varieties were revisited to find a standard monomial basis for the co-tangent space of the variety associated to these lattices. The singular locus and the multiplicities of the varieties associated to Bruhat lattices were discussed in \cite{b-l}\cite{GH}. In the same article interesting formulas for the multiplicities of these varieties at the distinguished points of the $T$ orbits were found. The authors in \cite{GH} also propose conjectures regarding the singularities of these for the general $I_{(d,n)}$ \cite{GH}.
In \cite{HL} the notion of a $\tau$ diagonal is introduced. The $\tau$-diagonals are particular class of diamonds that are incident on an embedded sublattice $D_\tau$ at only one point. In the present article we simplify that concept and introduce the set of diamonds $E_\a$. Based on a lower bound on the size of this set we hope to find the singular locus of the space $X(\L)$. As a combinatorial object $E_\a$ is interesting in it's own rights as counting the number of sublattices of a given distributive lattice is long standing complicated problem \cite{Dedikind} with various degrees of generalizations.

This article investigates the distributive lattice varieties that possess a tree as its poset of join-irreducibles( see for definition \cite{GG}). We provide a necessary and sufficient criterion for a distributive lattice to be a tree lattice. We also give an example for a tree lattice for which the singular locus of $X(\L)$ is non-empty. We define a particular type of tree lattice that is called a square lattice for which that set of join-irreducibles are union of chains except at the root. We show that these lattices are product lattices of chain lattices. We give an interesting tight bound on the cardinality of the number of diamonds containing a given element $\a$. We give a different proof of the fact that the singular locus of the varieties associated to square lattices is empty using elementary combinatorial arguments and the bound described before. We show that the affine cone $\widehat{X(\L)}$ over the variety $\widehat{X(\L)}$ has no singular points except at the vertex of the cone.

\section{Main Results}

\begin{thm}\label{a} Let $\L$ be a distributive lattice and $\a \in \L$. $Y(\mathcal{L})=V(I(\mathcal{L})) \subset \mathbb{A}^{|\L|}$ and let point $p_{\a}$ defined as \[ (p_{\a})_{\b} = \left\{ \begin{array}{ll}
         0 & \mbox{if $\b \neq \a $};\\
        c \neq 0 & \mbox{otherwise}.\end{array} \right. \] Where $(p_{\a})_{\b}$ denotes the $\b$th coordinate of the point $p_{\a}$. The point $p_\a$ is smooth if \[|E_{\a}| \geq |\mathcal{L}| -|J(\mathcal{L})|\]
\end{thm}

The above is a direct application of Jacobian criterion of smoothness for affine varieties. We use the inequality to prove that the restriction of the Jacobian matrix at the point $p_{\a}$ has rank more than or equal to $\card{E_\a}$ and therefore it is smooth at the points where $\card{E_\a}$ is more than the codimension of the variety in the full affine space at that point. And hence smooth by Jacobian criterion. See \cite{eis}.

\begin{thm}\label{b} For $\a \in \L$ where $L$ is a square lattice we have $|E_\a| \geq |\mathcal{L}| - |J| $
\end{thm}

We define an operation on distributive lattices which we call pruning. Given a maximal join irreducible $\b$ we find a sublattice $\L_\b$ which is without that join irreducible. For square lattices we have a simple structure of the poset $\L \setminus \L_\b$ which enables us to get a bound on the diamond relations with corners at $\a$ which are not in the sublattice $\L_\b$. This is an equality for a square lattice which is product of two chain lattices.

\begin{thm}\label{c} For a square lattice $\mathcal{L}$, $X(\mathcal{L})$ is non singular at all its points
\end{thm}
As a consequence of the theorem \ref{a} and \ref {b} we have the smoothness of the affine variety $\widehat{X(\L)}$ at all points except the point $p=(0,0,0 \ldots , 0)$ hence the projective variety $X(\L)$ is smooth at all its points.

\section{Definitions and Lemmas}
In this section we also recall the known results and basic definitions regarding distributive lattices, that will be used in this present article, a thorough text can be obtained in \cite{HL,GG,b-l, GLdef}. A partial ordered set $(P, \leq)$ is called a lattice if it is a non-empty set such that the two binary operations defined as $x \v y = inf\{ z \in P | z \geq x,y\}$ called ``join" of $x$ and $y$ and $x \w y = sup \{z \in P | z \leq x,y\}$ called ``meet" of $x$ and $y$ exist and are idempotent, associative and commutative \cite{GG} and for all choices of $x,y \in P$ they satisfy: \[x \w (x \v y)=x\] \[x \v (x \w y)=x\]

Further a lattice will be called a distributive lattice if it satisfies the distributive identity as defined below:
\begin{defn}\emph{Distributivity identity}
 $(x\vee y) \wedge z = (x \wedge z) \vee (y \wedge z)$.
\end{defn}
An element $x \in \D$ where $\D$ is a distributive lattice is called a join irreducible if it is not a join of two non-comparable lattice elements, or equivalently if $x =y \v z$ then either $x=y$ or $x=z$. The set of join irreducibles in the lattice $\D$ plays an important role, let us denote it by $J(\D)$ or simply $J$ if there is no scope of confusion. A subset $S$ of the poset $J$ is called a hereditary if $\forall x \in S$ and for all $y \leq x$ we have $y \in S$.

 \begin{defn} Two elements $\a,\b \in \D$ are said to be covers, or $\a$ covers $\b$ or $\a \gtrdot \b$ if $\a > \b $ and whenever there is $\a \geq x > \b$ then $ \a =x$.

  \end{defn}

  So a maximal chain $\mathcal{M}$ in $\D$ can be written as $\mathcal{M}= \{\underline{1}, a_1,a_2, \ldots , a_n , \underline{0} | \underline{1} \gtrdot a_1 \gtrdot a_2 \gtrdot \ldots \gtrdot a_n \gtrdot \underline{0},  \, a_i \in \D  \}$ where $\underline{1}, \underline{0}$ are the maximal and the minimal elements of $\D$ respectively. The important theorem by Birkhoff \cite{GG} should be quoted here.

\begin{thm}[G.Birkhoff \cite{GG}] The distributive lattice $\D$ and the lattice of hereditary sets of $J(\D)$ are isomorphic.
\end{thm}
Let us call the lattice of the hereditary subsets of $J$ by $\mathfrak{I}(J)$. The isomorphism in the above theorem is a basic tool for lattice theory and we will take the opportunity to put it in writing here. $\phi : \D \rightarrow$ $\mathfrak{I}(J)$ is defined as $\phi(\a)=\{\b \in J | \b \leq \a\}$. We will invent a shorter notation as $\phi(\a)=I_\a$. Observe that the subset $I_\a$ is a hereditary subset.  Note that the join and meet operations of the lattice $\mathfrak{I}$ are just union and intersection respectively. \label{lattice ideals}

To a distributive lattice $\D$=$\{ a_1, a_2, \ldots, a_m\}$ ( henceforth all lattices will be distributive unless mentioned otherwise ) we can attach the polynomial ring over a field $k$ as $k[\D]$ = $k [ x_{a_1},x_{a_2},x_{a_3} ,\ldots , x_{a_n}]$. In this polynomial ring we have the ideal $I(\D)$ generated by the set of binomials $\{x_\a x_\b -x_{\a \v \b} x_{\a \w \b}| \a,\b \in \D \, \mathrm{and} \, \a \nsim \b\}$ where $\a \nsim \b$ denotes that they are non-comparable elements of $\D$. This ideal is of some interest to both geometers and lattice theorists as we see in \cite{GLdef,Hibi,HL} these ideals related to distributive lattices are discussed in various contexts. The vanishing locus of the the ideal $I(\D)$ in the affine space $\mathbb{A}^{|\D|}$ is discussed in \cite{HL}. The singular locus of these algebraic varieties are of considerable interest \cite{HL,b-l}. In this present article we will define a class of distributive lattices for which the vanishing locus of the ideal $I(\D)$ in the projective space $\mathbb{P}^{\card{\D}-1}$ is non-singular. For an introduction to toric varieties reader may consult \cite{F} and \cite{toroidal} and \cite{ES}.

In \cite{HL} we have seen that the dimension of the variety $\widehat{X(\D)}=V(I(\D))\subset \mathbb{A}^{\card{\D}}$ is given by the number of join irreducible elements in the lattice $\D$ which is equal to the length of a maximal chain of $\D$ \cite{HL}. A chain in a lattice $\D$ is a totally ordered subset of the lattice $\D$, and a "chain lattice" is a lattice which is totally ordered. Note that given a natural number $n$ there is a unique chain lattice up to a lattice isomorphism let us call that lattice $c(n)$. Let us also write down the following definition in this juncture.

\begin{defn}
For a distributive lattice $\D$ we define $\dim(\D)= \card{\mathcal{M}}$. The natural number $\card{\mathcal{M}}$ is also called the length of the maximal chain $\mathcal{M} \subset \D$.
\end{defn}


In this section we introduce few definitions regarding distributive lattices that will be used in proving the results of the present article.
\begin{defn}\label{square lattice}
A distributive lattice $\mathcal{L}$ is called a \emph{tree lattice} if the poset of join irreducible elements of the lattice $\mathcal{L}$ is a tree.
\end{defn}
The motivation behind the previous definition is the observation that the distributive lattices that possess a tree as join irreducible poset is easier to handle.

\begin{defn}
A distributive lattice $\L$ will be called an \emph{honest lattice} if for every $\alpha$, $\beta$ $\in J(\mathcal{L})$ such that $\alpha$ covers $\beta$ in $J(\L)$ then $\alpha$ covers $\beta$ in $\L$.
\end{defn}

The above definition is motivated by the fact that for such a lattice the set of join irreducible is a nicely embedded sub-lattice.

The following theorem gives a clear picture on the above two definitions.
\begin{thm}\label{tree honest equivalence}
A distributive lattice is a tree lattice if and only if it is a honest lattice.
\end{thm}
\begin{proof}
Let $\L$ be an honest lattice and let $J=J(\L)$ denote its set of join irreducible. If $J$ is not a tree then there exist $\alpha$, $\beta$, $\gamma$ , $\delta$ $\in J$ such that $\gamma$ covers both $\alpha$ and $\beta$ in $J$. Further $\delta$ is covered by both $\alpha$ and $\beta$ in $J$. In that case $\delta \leq \alpha \vee \beta \leq \gamma$ which means $\alpha$ is covered by $\alpha \vee \beta $ covered by $\gamma$ since $\alpha \vee \beta$ is not a join irreducible it leads to a contradiction to our assumption that $\L$ is honest.

For the reverse direction let us assume that $\L$ is a square lattice, then if its not honest there are elements $\alpha$, $\beta$ in $J$ such that $\alpha \gtrdot \beta$ in $J$ but not in $\L$. Which means there is an element $\theta \in \L$ such that $\alpha > \theta > \beta$ in $\L$. And since $\theta \notin J$ there are elements $x$,$y$ $\in \L$ such that $\theta = x \vee y$. Now since $x \vee y < \alpha$ both $x$ and $y$ are less than $\alpha$. $\Rightarrow x \wedge y < \alpha$. Since $x \vee y > \beta \Rightarrow $ either $x $ or $y$ is larger than $\beta$ from lemma \ref{check}.

Without  loss of generality, let us assume that $x$ is larger than $\beta$. Since $x$ is not in $J$ we will have two elements $x_1$, $y_1$ $\in \L$ such that $x_1 \vee y_1 = x > \beta$. Using the lemma below we have $x_1 > \beta$ and hence we can continue the whole process, but our lattice $\L$ being a finite lattice, the process will ends after finitely mane iterations. Which mean we will have $z$ such that $\alpha > z >\beta $ and $z \in J$ contradicting our assumption.
\end{proof}

Observe that in a tree lattice every join irreducible has a unique join irreducible below it since the join irreducible poset is a tree. This is a very important property that will be exploited in this article. Because of the importance of the fact we will write it down as a lemma.

\begin{lem}\label{unique}
Let $\b $ be a join-irreducible in a tree lattice $\L$, then there is a unique join irreducible $\g \in J(\L)$ such that $\b \geq \g$ and they are covers in $J$.
\end{lem}
\begin{proof}
Clearly if there are more than one such join irreducible say $\l_1,\l_2$ then there are two paths from the root $\rho$ to $\b$ via each of $\l_1,\l_2$ which contradicts the hypothesis that the lattice is a tree lattice.
\end{proof}

\begin{lem}\label{check}
If $x$, $y$ $\in \L$ such that $x \vee y > \beta$ where $\beta \in J$ then either $x > \beta $ or $y > \beta $.
\end{lem}
\begin{proof}
Since $I_{x \vee y} = I_x \cup I_y $ which contains $I_{\beta}$ which correspond to the element $\beta$ $\Longrightarrow \beta \in I_x$ or $I_y$.
\end{proof}

\begin{defn}\label{non singular lattice}
A distributive lattice $\L$ will be called a non singular lattice if the projective variety $X(\L)$ is smooth.
\end{defn}

\begin{defn} \label{pruning}
A pruning of a lattice $\L$ with respect to a maximal join irreducible $\beta$ is the lattice $\mathcal{I}(J\setminus \{\beta\})$ where $\mathcal{I}(A)$ is the poset of hereditary subsets of a poset $A$.
\end{defn}

Let us also invent a notation for a pruning, $\L_\b$=$I( J \setminus \{\beta \})$

At this juncture we write down an idea related to the joining move that we defined above even though it will not be used in this present paper but for the sake of a completeness of mind.

\begin{lem} \label{chain}
 let $\b$ be a maximal join-irreducible element in the square lattice $\L$ then $B_\b=\L \setminus \L_\b$ is the set $\{ \g \in \L | \g \geq \b\}$.
\end{lem}
\begin{proof}
The set $B_\b$ can be identified with the set $\{ \g \in \L | \g \geq \b\}$ because since by pruning we have deleted the maximal join irreducible $\b$ from all ideals that contains it, we are left with the ideals that do not contain the element $\b$ in $\L_\b$. Which means that the elements of $B_\b$ are the ideals in $J(\L)$ that contains the element $\b$ or equivalently these are precisely the elements $\{\g \in \L | \g \geq \b\}$.
\end{proof}

\begin{lem} \label{bijection}
Let $\b \in \L$ be a maximal join irreducible and let $\b_1 < \b$ be the unique (see \ref{unique}) join irreducible below $\b$ then the posets $B_b=\L \setminus \L_\b$ and $B_{\beta_1}= \mathcal{L}_\beta \setminus (\mathcal{L}_\beta)_{\beta_1}$ are isomorphic as posets.
\end{lem}

\begin{proof}
Let us define a poset map $\phi: B_{b_1} \longrightarrow B_\b$ as $\phi(x)=x \v \b$. It is clear that this a poset map. So let us see why it is a bijection. If $\phi(x)=\phi(y)$ then $x \v \b = y \v \b$ which means $I_x \cup \{b\}=I_y \cup \{\b\}$ But since $I_x$ and $I_y$ does not contain $\b$ we have $I_x=I_y$, or by Birkhoff's theorem we have $x=y$. For surjectivity part of the claim, note that if $\l \in B_b$ then $\l \geq \b$ now if we consider the set $I_\l \setminus \{\b\}$ this is an ideal in $B_{\b_1}$, for if $s \in I_\l \setminus \{\b\}$ and if $t \leq s$ then since $I_\l$ is an ideal $ t \in I_\l$ hence $t \in I_\l \setminus \{\b\}$. Also note that $\b_1 \in I_\l$ since $\b_1 \leq \b$ and hence it is in $I_\l \setminus \{\b\}$. Hence putting all these together we get an element $\g \in B_{\b_1}$ such that $\g \v \b = \l$ completing the surjectiveness of the claim.
\end{proof}

The algebraic variety associated to a tree lattice need not be smooth always. Let us explain the situation with the help of an example. The lattice in the figure below the point $p=(0,0,1,0,0,0,0,0,0,0)$ is not a smooth point. Even though the lattice is a tree lattice. 
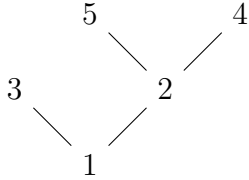
\label{example}
\begin{center}

\begin{tikzpicture}
  \node (max) at (7,4) {10};
  \node (a) at (6,3) {9};
  \node (b) at (7,3) {8};
  \node (c) at (8,3) {7};
  \node (d) at (6,2) {6};
  \node (e) at (7,2) {5};
  \node (f) at (8,2) {4};
  \node (g) at (5,1) {3};
  \node (h) at (7,1) {2};
  \node (i) at (6,0) {1};
  \draw (max) -- (a) -- (d) -- (g) -- (i) -- (h) -- (f) -- (c) -- (max) ;
  \draw (h) --(d) ;
  \draw (h) -- (e) -- (a);
  \draw (d) -- (b) -- (max);
  \draw (e) -- (c);
  \draw (f) -- (b);
\end{tikzpicture}

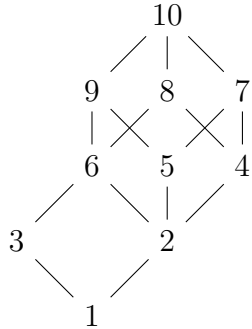
Related to the join irreducible poset

 \begin{tikzpicture}
  \node (max) at (0,1) {$3$};
  \node (a) at (1,0) {$1$};
  \node (b) at (2,1) {$2$};
  \node (c) at (1,2) {$5$};
  \node (d) at (3,2) {$4$};
  \draw (max) -- (a) -- (b) -- (c) ;
  \draw (b) -- (d);
\end{tikzpicture}

\end{center}

Note that in the above picture the lattice corresponding to $\mathcal{L}$ is not non-singular as the point given by the coordinates $p=(0,0,1,0,0,0,0,0,0,0)$ is not smooth. Which means the projective space \mbox{Proj$k[\D]$} is not smooth.

This above example motivates us to restrict our lattice further to what we will call a square lattice.

\begin{defn} A tree lattice $\L$ will be called a square lattice if the graph of the Hasse diagram of its poset of join-irreducibles has the following property. Degree of all the vertices except for the root is at the most two.
\end{defn}

From the above definition we will derive the following properties of the square lattices. Observe that for a square lattice $\L$ the poset of join irreducibles is union of a collection of chains almost disjoint.
\begin{lem}
The join irreducible poset $J$ of a square lattice $\L$ is union of a collection of chains, disjoint except at the root.
\end{lem}
\begin{proof}
If $\a \in J$ and $\a$ is not the root say $\rho$ then it has degree at most two in the Hasse graph. Which means if $\a \in J$ is a maximal element and if it is different from the root then it has a unique predecessor $\a_1$. And the element $\a_1$ has one predecessor or it is the root. So continuing the argument and observing that we are dealing with finite lattices we get a chain $\{ \a,\a_1,\a_2 , \ldots , \a_t = \rho\}$ call this chain $c(t+1)_\a$. Next let us pick another maximal element $\b \in J$, if there are none then we have successfully written the join irreducible set as a union of chains, else with the maximal element $\b$ we can associate another chain $c(t_1)$ and this chain is disjoint with the chain $c(t)$ except at $\rho$. In other words $c(t) \cap c(t_1) =\{\rho \}$. Continuing the process we will get the desired result.
\end{proof}

\begin{thm} \label{square} There are natural numbers $n_1,n_2,\ldots , n_t$ such that the square lattice $\L$ is isomorphic to the product of chain lattices $c(n_1),c(n_2), \ldots, c(n_t)$.
\end{thm}
\begin{proof}
 Since we have a square lattice the poset of join-irreducibles $J$ can be written as union of chains $c(n_1),c(n_2), \ldots , c(n_t)$. We will look at the lattice $\L$ as the lattice of hereditary sets of $J$. And let us give a map as: \[ \phi : \mathcal{L} \longrightarrow \prod{} c(n_i) \] as \[ \phi ( I_\a)= (\b_1,\b_2,\b_3,\ldots , \b_t )\]

 Where $\b_i$ is the maximal element of the chain $ I_\a \cap c(n_i)$. Let us see that this map is surjective. If we choose a collection of elements $\b_1, b_2, \ldots \b_t$ we look at the hereditary set $I=\{ z \in \mathcal{L} | \exists \, i \, \mathrm{such \, that } \, z \leq  \beta_i \}$ surely by construction the elements $\b_i$ , $i \leq t$ are the maximal elements of the set $I \cap c(n_i)$. Hence the map defined as above is surjective.

 Now for the injectivity if we have two elements $\a,\g \in \L$ such that $\phi(\a)=\phi(\g)=(\b_1,\b_2, \ldots \b_t)$ then $\b_i$ is the maximal element of both $I_\a \cap c(n_i)$ and $I_\g \cap c(n_i)$ but since $I_\a$ and $I_\g$ are ideals we have $I_\a \cap c(n_i) =I\g \cap c(n_i) \, \forall i$ or equivalently $I_\a=I_\g$ or $\a=\g$.
\end{proof}



In this section we prove a general theorem about singularity of lattice toric varieties. First let us formulate the notion of a diamond and a diamond relation precisely.
\begin{defn}\label{diamond}
A diamond in a distributive lattice $\mathcal{L}$ is $D=\{\a,\b,\gamma,\d\}$ such that there are two non comparable elements $x,y \in D$ and $x \vee y , x \w y \in D$.
\end{defn}
Note that every diamond $D=\{\a, \b ,\a \vee \b , \a \wedge \b \}$ gives rise to a generator in the ideal $I(\L)$ namely the relation ( henceforth will be called a diamond relation ) $f_{D}=x_{\a}x_{\b}-x_{\a \vee \b }x_{\a \wedge \b}$. So we can write the ideal $I(\L)$ as the ideal in $k[\L]$ generated by the relations $f_{D}$ where $D$ ranges over all the diamonds of the distributive lattice $\L$. Let us also call the set of all the diamond relations of the lattice $\L$ as $\mathfrak{D}$ Given $\a \in \L$ let us define a class of diamond relations that contains $\a$.
\begin{defn}
For $\a \in \L$ , $E_{\a}=\{(\a,\g)| \exists D \in \mathfrak{D} | \a, \g \in D\}$
\end{defn}

\begin{lem}\label{inequality} For a square lattice $\L$ and $\b$ a maximal join-irreducible in $J(\L)$ we have $$ \card{E_\a \setminus E_\a(\L_\b)}= \card{E_\alpha} -\card{E_\alpha(\L_\b)} \geq \card{\L} - \card{\L_{\b}} - 1 $$
\end{lem}
\begin{proof}
Let $n=\card{\L}-\card{\L_\b}$, and let $\{b_1,b_2,\ldots , b_n\}=\L \setminus \L_\b$, further let us assume that $\{D_1,D_2, \ldots , D_r\}=E_\a \setminus E_\a(\L_\b)$. Let us call $D_i=\{\a, x_i,y_i,z_i\}$. Since these diamond relations are not in $E_\a(\L_\b)$ which means at least one of $x_i,y_i,z_i \in \L \setminus \L_\b$ for each $i \leq r$ without loss of generality let us assume that it is $x_i$. Now we know from \ref{chain} that $\L \setminus \L_\b$ is given by $\{\g \in \L | \g \geq \b\}$. And since $x_i \in \L \setminus \L_\b$ we have $x_i \geq \b$ which leads us to the following two cases:
\begin{itemize}
\item \emph{\underline{case one}} $x_i$ is the maximal element of the diamond $D_i$. In this case since $x_i$ is the maximal element it is join of two other elements $\d,\e \in D_i$ which means at least one of $\d,\e \geq \b$ see \ref{check}. But it cannot be both since that will imply $\a \geq \b$ which is contrary to our assumption that $\a \in \L_\b$. So we have exactly two elements of $D_i$ in $\L \setminus \L_\b$.
\item \emph{\underline{case two}} $x_i$ is not the maximal element. In this case we have another element $\d \in D_i$ such that $\d \geq x_i \geq \b$. But only these two elements are larger than $\b$ since otherwise we will have $\a > \b$ and we contradict the assumption about $\a \in \L \setminus \L_\b$.
\end{itemize}

 To sum up the above two cases we can say that we have one of $y_i,z_i \in \L \setminus \L_\b$ apart from $x_i$ which is by our assumption is in the chain. Or one can say that there exist $b_{i_1} , b_{i_2}$ such that $x_i =b_{i_1}$ and either of $y_i,z_i$, without loss of generality let us assume $y_i = b_{i_2}$. Let us rewrite $D_i$ in light of this new information. $D_i=\{\a, b_{i_1},b_{i_2},z_i\}$. Let us also see that $b_{i_1}$ $ b_{i_2}$ are comparable since otherwise if these are the non-comparable elements in the diamond $D_i$ then both being larger than $\b$ it will imply both $b_{i_1} \v b_{i_2}$ and $b_{i_1} \w b_{i_2}$ larger than $\b$ but since one of these two must be $\a$ which is not larger than $\b$ we lead to a contradiction to our assumption that the elements $b_{i_1}$ and $ b_{i_2}$ are non-comparable. So without loss of generality let us assume that $b_{i_1} > b_{i_2}$. Note that with these assumptions in place we see that $b_{i_1}$ is the maximal element of the diamond $D_i$. So the Hasse diagram of the diamond looks like either of the following two:
 
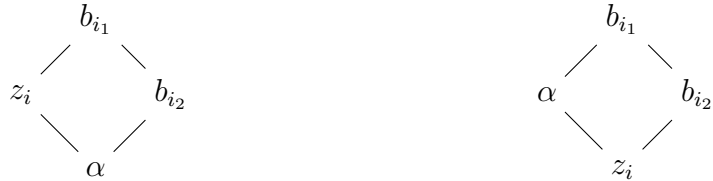
\begin{center}
\begin{tikzpicture}
  \node (max) at (1,2) {$b_{i_1}$};
  \node (a) at (2,1) {$b_{i_2}$};
  \node (b) at (0,1) {$z_i$};
  \node (c) at (1,0) {$\a$};

  \draw (a) -- (max) -- (b) -- (c) -- (a) ;

  \node (max) at (8,2) {$b_{i_1}$};
  \node (a) at (9,1) {$b_{i_2}$};
  \node (b) at (7,1) {$\a$};
  \node (c) at (8,0) {$z_i$};

  \draw (a) -- (max) -- (b) -- (c) -- (a) ;
\end{tikzpicture}

\end{center}
Now by \ref{greater} we have $\card{E_\a}-\card{E_\a(\L_\b)} \geq \card{\L}-\card{\L_\b}-1=r-1$.

\end{proof}

For a square lattice $\mathcal{L} = \prod_{i \leq r } C(n_i)$, we know that the set of join irreducibles can be identified with the poset $\cup C(n_i)$, let us take a maximal join irreducible $[\beta]$ in $\L$ , given by $[\beta]=(\rho,\rho,\rho,\ldots, \b,\rho, \ldots, \rho)$ where $\b$ is the maximal element of $C(n_i)$ and $\rho$ is the minimal element of $\L$. Note that since the join irreducibles of $\L$ is identified with the union of $C(n_j)$ we will call this join irreducible $[\beta]$ by $\b$ that way we will reduce the burden of notation without hampering the generality of the treatment. Let us prove the lemma below with these notations in mind.

\begin{lem} \label{greater} For $\L$, $\b$, $C(n_j)$, $B_\b$ as above let $\a \in \L_\b$ then $\card{E_\a} - \card{E_\a(\L_\b)} \geq \card{B_\b} -1$

\end{lem}

\begin{proof}
Let us write $\a$ as the tuple $(a_1,a_2,a_3, \ldots a_r)$ and without loss of generality let us assume that $[\beta]=(\b,\rho,\rho \ldots , \rho)$. Let us call $b = min \{ \g \in B_\b | \g \geq \a \}$. So we can write $b$ even more explicitly with these information in place as $b= ( \b , a_1,a_2 ,\ldots , a_r)$. Given $\b_1=(\b,\g_2,\g_3,\ldots,\g_r) \in B_\b$ and $b_1 \neq b$ (continuing with the same notations as in the previous lemma ) we will show that there is a diamond $D_{\b_1}=\{\a,b_1,\b_2,z\}$ such that $D_{\b_1}$ and $D_{{b_1}'}$ are different for different elements $b_1$ and ${b_1}'$.  We will prove the above by showing the following cases:
\begin{itemize}
\item \emph{\underline{Case one}} $b_1 >b$ : in this case note that $\g_i \geq a_i \, , \forall i \geq 2$. Let us take $z=(a_1,\g_2,\g_3,\ldots , \g_r)$. We have $z \v b = b_1 $ and $z \w b = \a$
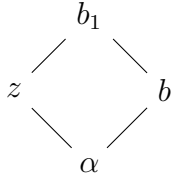
\begin{center}

\begin{tikzpicture}
  \node (max) at (1,2) {$b_1$};
  \node (a) at (0,1) {$z$};
  \node (b) at (2,1) {$b$};
  \node (c) at (1,0) {$\a$};

  \draw (a) -- (max) -- (b) -- (c) -- (a) ;
\end{tikzpicture}
\end{center}
\item \emph{\underline{Case two}} $b_1 < b$: in this case we have $\g_i \leq a_i \, , \forall i \geq 2$. Note that $\a$ and $\b_1$ are non-comparable since otherwise if $b_1 >\a$ that contradicts the assumption about the minimality of $b$ and if $b_1 < \a $ then $\a \in B_\b$ contradicting our assumption about $\a$. And since these are non-comparable we take $z= \a \w \b_1$. Note that we have $b_1 \v \a = b$ also.
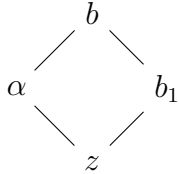
\begin{center}

\begin{tikzpicture}
  \node (max) at (1,2) {$b$};
  \node (a) at (0,1) {$\a$};
  \node (b) at (2,1) {$b_1$};
  \node (c) at (1,0) {$z$};

  \draw (a) -- (max) -- (b) -- (c) -- (a) ;
\end{tikzpicture}
\end{center}

\item \emph{\underline{Case three}} $b_1$ is non-comparable to $b$ : in this case we see that $b_1$ is non-comparable to $\a$ since otherwise if $b_1 > \a$ then $b > b_1 \w b > \a$ and $b$ strictly larger than $b_1 \w b$ so it contradicts the minimality of $\a$.  And if $b_1 < \a$ then it contradicts the definition of $\a$. So $b_1 $ and $\a$ are non-comparable so let us take $z = b_1 \w \a$ and $b_2 = b_1 \v \a$.

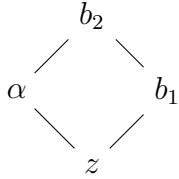
    \begin{center}

\begin{tikzpicture}
  \node (max) at (1,2) {$b_2$};
  \node (a) at (0,1) {$\a$};
  \node (b) at (2,1) {$b_1$};
  \node (c) at (1,0) {$z$};

  \draw (a) -- (max) -- (b) -- (c) -- (a) ;
\end{tikzpicture}
\end{center}
\end{itemize}

So the diamonds that we have specified to every element in $B_\b \setminus \{b\}$ are all distinct since given $b_1$ we have the tuple $(\a,b_1) \in E_\a$, which means that the set $E_\a \setminus E_\a(\L_\b)$ has cardinality more than or equal to $\card{B_\b}-1$. This completes the proof of the lemma.
\end{proof}

\section{Proof of the Main Results} \label{main}

Now we are ready to prove the main theorems of the article.

\subsection{ Proof of The Theorem \ref{a}}

\begin{proof}
Let us write down the set of diamonds in the lattice $\D$ as $D_1,D_2 , \ldots, D_m$ of which let us choose $D_1, D_2, \ldots, D_r$ diamonds containing the tuples $(\a_i,\g_i) \in E_\a$, as there can be more than one diamonds containing the same tuple. Let us also enumerate the diamond relations that generate the ideal $I(\L)$ as $f_1,f_2, \ldots, f_m$ where $f_i = f_{D_i}$. So for each $f_i$, $i \leq r$ we have a $\b_i \in \L$ such that $f_i= x_{\a}x_{\b_i} - x_{\d_i}x_{\gamma_i}$ for some $\d_i , \gamma_i \in \L$.
So we have $\frac{\partial f_i }{\partial x_{\b_i}}|_{p_\a}=x_\a|_{p_\a}=c \neq 0 , \, \forall i \leq r $ and $\frac{\partial f_i}{\partial x_\b}=0 , \, \forall i > r$. The last equality is derived from the fact that $x_\a$ is the only nonzero coordinate and it only occurs in the diamond relations $f_i$ for $i \leq r$. So the Jacobian matrix \cite[p.~31]{Ha}\cite[p.~404]{eis} has the following shape.

\[ \left( \begin{array}{c}
\frac{\partial f_i}{\partial x_{\beta_j}}
 \end{array} \right)= \left(
\begin{array}{cc}
cI_{r \times r} & 0 \\
A & B
\end{array}\right)\], for some appropriate size matrices $A$ and $B$. So the rank of the Jacobian matrix is greater than or equal to $r=|E_\a|$ and since the dimension of the variety $X$ at the point $p_\a$ is $|J(\L)|$ we have the result \cite[p.~32]{Ha}.

\end{proof}

\subsection{Proof of the Theorem \ref{b}}

\begin{proof}
To prove the statement we will make use of the pruning operation.

We will prove the lemma by induction on $|J|$. For the base case $|J|=1$ the lattice $\L$ is just a chain. Hence $E_\a=\emptyset $ and since $J=\L$ in this case we have the result that we seek to prove. For the general case we will break it into two cases.
\begin{itemize}
\item \emph{\underline{Case One}}
 Since $\a$ is not the maximal element of the lattice $\L$ we can always find a join irreducible $\b$ such that $\b \nless \a$. We prune the lattice $\L$ with respect to the maximal join-irreducible $\beta$ and let the pruned lattice be called $\L_\b$ and let us call the join irreducibles of this sublattice be $J_\b$.
 By induction hypothesis we know that \[ \card{E_\alpha(\L_\b)} \geq \card{\L_\b} - \card{ J_\b} \]
Now we also have \[ \card{J_b} = \card{ J} - 1 \] Putting these two equations together we have \[\card{E_\a (\L_{\b})} \geq \card{\L_{\b}} - \card{ J} + 1 \]

Note that if we prove the following \[ \card{E_\alpha} -\card{E_\a(\L_{\b})} = \card{\L} - \card{\L_{\b}} - 1 \] then we have
\[ \card{E_\a}-\card{E_\a(\L_\b)} \geq \card{\L} - \card{L_\b}-1\]
\[ \Leftrightarrow \card{E_\a} \geq \card{\L}-\card{\L_\b} + \card{E_\a(\L_\b)} -1 \]
\[ \geq \card{\L} - \card{\L_\b} + \card{\L_\b} - \card{J} +1 -1 \]
\[ \geq \card{\L} - \card{J}\]

If we prove the above inequality we will have proved this particular case of the lemma. We conclude the above inequality in the following lemmas \ref{inequality}

\item \emph{\underline{Case Two}} In this case we have $\a = \mathrm{max}(\L)$. We reduce this case to the previous case by replacing the lattice $(\L, \leq)$ with the lattice $(\L , \triangleleft)$ where the order $\triangleleft$ is given by the following rule: $x \triangleleft y \Leftrightarrow y \leq x $. Observe that a diamond in the lattice $\L,\leq$ is still a diamond in $(\L, \triangleleft)$ and vice versa. And also observe that the set $E_\a$ remains same for both the lattices for a given element $\a \in \L$. But since we have reversed the order the maximal element $\a$ is now the minimal element of the lattice $(\L, \triangleleft)$. Hence by the previous case we have the required inequality.

\end{itemize}
\end{proof}

\subsection{Proof of the Theorem  \ref{c}}

\begin{proof}
We will prove that the affine cone over the variety $X(\L)$, namely $\widehat{X(\mathcal{L})} =\mathrm{Spec}( k[\L])$ is smooth at all points except at the vertex. Hence the projective variety $X(\L)$= $\mathrm{Proj} ( k[\L])$ is smooth at all points. Let $p \in \widehat{X(\L)}$ which is not the origin, so we have at the least one $\a \in \L$ such that the \mbox{$\a$th} coordinate of $p$ namely $(p)_\a = x_\a$ is nonzero. Now by theorem \ref{a} we know that the point $p$ is smooth if $\card{E_\a} \geq \card{\L} - \card{J}$ which is true for any $\a$ for a square lattice $\L$ by theorem \ref{b}.

\end{proof}


\section{References}

\bibliographystyle{abbrv}
\bibliography{main}

\end{document}